%% file: theta.tex
\documentclass{amsart}

\input{preamble}

\begin{document}

\title{Efficient computation of non-archimedean Theta functions}

\author{Marc Masdeu}
	\address{Departament de Matem\`atiques\\Universitat Aut\`onoma de
		Barcelona\\08193 Bellaterra, Barcelona, Catalonia}
\email{marc.masdeu@uab.cat}
\thanks{The first named author thanks Steffen Müller and Enis Kaya for very fruitful discussions at the
beginning of this project.}

\author{Xavier Xarles}
	\address{Departament de Matem\`atiques\\Universitat Aut\`onoma de
		Barcelona\\08193 Bellaterra, Barcelona, Catalonia}
\email{xavier.xarles@uab.cat}
\thanks{Both authors were supported by research project  PID2020-116542GB-I00}


\date{}

\dedicatory{}

\begin{abstract}
  We describe an efficient iterative algorithm for the computation of theta functions of
  non-archimedean Schottky groups and, more generally, of (non-archimedean) discontinuous groups.
\end{abstract}

\maketitle

\section{Introduction}
Let $K$ be a field complete with respect to a non-trivial non-archimedean absolute value.
We will denote by $|\ |$ the absolute value, and by $R = \{a\in K\mid |a|\le 1\}$
its ring of integers.

Let $\Gamma\subset \PGL_2(K)$ be a (non-archimedean) Schottky group,
with fixed generators $\gamma_1,\ldots,\gamma_g$ ($g > 1$).
We will write $\gamma_{-i}$ for $\gamma_i^{-1}$.
Denote by $\Sigma\subset \PP^1(K)$ the ``bad'' points for $\Gamma$
(the closure of the set of limit points), and set $\Omega=\PP^1(K)\setminus \Sigma$. The group
$\Gamma$ acts discontinuously on $\Omega$, and the quotient $\Omega / \Gamma$ has the
structure of a rigid-analytic space, which is the analytification of a curve $C$ over $K$
of genus $g$, known as a \emph{Mumford curve} (see \cite{GvdP} for more details).

One of the main tools to study such curves $C$ is the theory of theta functions. One way of
thinking about the theta function attached to $\Gamma$ is as an analytic function $\Theta_\Gamma(z;a,b)$
which depends on parameters $a,b\in \Omega$, and which is defined as
\[
\Theta_\Gamma(z;a,b) = \prod_{\gamma\in\Gamma}\frac{z-\gamma a}{z-\gamma b}.
\]
Using this function and its derivative one
can compute the canonical embedding of $C$, or the period matrices of their Jacobian, and in~\cite{MR}
it is shown how this is done in practice.

The obvious algorithm to approximate the value of $\Theta_\Gamma(z;a,b)$, which has been used
in the literature, consists on
computing a finite product over all matrices $\gamma\in \Sigma$, where $\Sigma$
is the set of all elements in $\Gamma$ of length up to some bound which depends linearly
on the required precision.
However, a direct count shows that there are exactly
$\frac{g(2g-1)^n -1}{g-1}$ words of length up to $n$, which implies that any
algorithm that needs to enumerate these words will be of exponential running time.

When the Schottky group arises from the units of an order in a quaternion algebra, several authors
have devised polynomial-time algorithms, called ``overconvergent'', to compute $\Theta_\Gamma(z;a,b)$ by describing the elements
of the group iteratively (\cite{negrini-masters}) or using the action of a
Hecke operator (\cite{darmon-vonk}, \cite{guitart-masdeu-xarles}).

In this note, we propose a polynomial-time algorithm that finds an analytic function  on $z$ approximating
$\Theta_\Gamma(z;a,b)$. Once this function is obtained,
evaluating it at any argument $z\in\Omega$ can be done in time linear in the precision. In the quaternionic
setting, we expect that
the efficiency of our algorithm is comparable to the existing ``overconvergent'' methods.

The basic idea consists in decomposing $\Theta_\Gamma(z;a,b)$ as a product
\[
\Theta_\Gamma(z;a,b) = \prod_{n=0}^\infty \Theta_{n}(z;a,b),
\]
with $\Theta_n(z;a,b)$ concerning only elements of length $n$. If one further decomposes $\Theta_n(z;a,b)$ as a product
\[
\Theta_n(z;a,b) = \prod_{i=\pm 1}^{\pm g}\Theta_n^{(i)}(z;a,b)
\]
where in $\Theta_n^{(i)}$ only words starting with $\gamma_i$ are considered, then one can describe $\Theta^{(i)}_{n+1}(z;a,b)$
in terms of all the $\Theta^{(j)}_n(z;a,b)$ by a simple recurrence. These basic functions $\Theta^{(i)}_n(z;a,b)$ turn out
to be elements of an affinoid algebra and are amenable to computation, thanks to the existence of the so-called
good fundamental domains for the Schottky group $\Gamma$. In~\cite[Section 4]{MR} the authors describe an algorithm that
computes a good fundamental domain from only the data of a generating set for the group $\Gamma$, so this can be
thought of as already computed data. It is worth pointing out, however, that depending on how far are the given
generators from being in good position, one may have to enumerate many words just to construct the good fundamental domain.
The naive algorithm would need correspondingly more iterations as well, so we will not take this problem into account and
will assume the good fundamental domain has already been computed.

The paper is organized as follows. In Section~\ref{sec:basics} we review some basic facts on Schottky groups
and their fundamental domains. Section~\ref{sec:theta} introduces the Theta function in a way more suitable to
work with. We give precise estimates for its convergence, which are crucial for the algorithm to produce provably
correct results. In Section~\ref{sec:algorithm} we describe the iterative algorithm and prove all the results
needed to show correctness. We have implemented the algorithm in Sage (\cite{sagemath}) and we end
with Section~\ref{sec:examples} where we illustrate the algorithm in some example cases, for which we compare the
running time with a naive implementation.

\section{Basics on Schottky Groups}
\label{sec:basics}
\begin{definition}
  A matrix $\gamma\in \PGL_2(K)$ is hyperbolic if it has two eigenvalues
  with distinct absolute value.
\end{definition}

Recall that a subgroup $G \leq \PGL_2(K)$ is \emph{discontinuous} if its set of limit points is not all of
$\PP^1(\mathbb{C}_K)$
and the orbits any point in $\PP^1(\mathbb{C}_K)$ has a compact closure.

\begin{proposition}
  Let $\Gamma\leq \PGL_2(K)$ be a subgroup. Then the following are equivalent:
  \begin{enumerate}
    \item $\Gamma$ is finitely generated, and every
    non-identity $\gamma\in \Gamma$ is hyperbolic.
    \item $\Gamma$ is finitely generated, discontinuous and every non-identity $\gamma \in \Gamma$ has infinite order.
    \item $\Gamma$ is free, discrete and finitely generated.
  \end{enumerate}
\end{proposition}
\begin{proof}
  See \cite[Section I.1]{GvdP}.
\end{proof}

\subsection{\texorpdfstring{Balls in $\PP^1$ under the action of $\PGL_2(K)$}
  {Balls in P1 under the action of PGL2(K)}}

Consider the set
\[
  B(a,\rho) =\{z\in K \mid |z-a| < \rho\},
\]
for $a\in K$ and $\rho\in |\overline{K}^\times|$. This will be a called a (proper) ball
in $\PP^1_K(K)$. We will denote also by $B(a,\rho)$ the corresponding affinoid
in $\PP^1_K$, whose $K$-rational points are the elements of the ball.
We will denote by
\[
  B(a,\rho^+)=\{z\in K \mid |z-a|\leq \rho\},
\]
and
\[
  \partial B(a,\rho)=\{z\in K \mid |z-a|=\rho\}.
\]

It is useful to have a formula for the image of a ball under a Möbius transformation.

\begin{lemma}
  \label{lemma:action-balls}
  Let $g\in \PGL_2(K)$, and let $B$ be a proper ball (either open, or closed) of radius $r$
  and let $P$ be a point in $B$. Then:
  \[
    g B = \begin{cases}
      B\left( g(P), |g'(P)|r\right) & \text{if } g^{-1}(\infty) \notin B \\
      \PP^1\setminus B\left( g(\infty), \frac{|g'(\infty)|}{r}\right) & \text{else.}
    \end{cases}
  \]
  Here $g'(P) = \frac{ad-bc}{(cP+d)^2}$ if $g=\smtx abcd$,
  and $g'(\infty) = \frac{ad-bc}{c^2}$. Also, one has to take the open or closed
  form of the resulting ball as appropriate.
\end{lemma}
\begin{proof} 
First, observe that the condition $g^{-1}(\infty) \notin B$ is equivalent to  $|cP + d| > r|c|$.
In this case, we have that for any  $Q\in B$, $|g(P)-g(Q)|=|g'(P)||P-Q|$, since $|cP + d| = |cQ + d|$.
This shows that $g(Q)\in B\left( g(P), |g'(P)|r\right)=:B'$, hence $ g B \subset B'$.
We apply the same argument for $g^{-1}$, the ball $B'$ and the point $g(P)$, as $g(\infty)=a/c\notin B'$,
since 
$$ |g(P)-g(\infty)|=\left|\frac{aP+b}{cP+d}-\frac{a}{d}\right|=|g'(P)| \frac{|cP + d| }{|c|}>|g'(P)| r.$$
We get that, for any $Q'\in B'$,
$$|P-g^{-1}(Q')|=|(g^{-1})'(g(P))||g(P)-Q'|\le |(g^{-1})'(g(P))||g'(P)| r=r,$$ 
which shows that $ B \supset g^{-1}(B')$.

For the second case (that of $g^{-1}(\infty) \in B$), observe that $B = B(-d/c, r)$. Write now
\[
\frac{az+b}{cz+d}-\frac a c = \frac{1}{c}\frac{bc-ad}{cz+d},
\]
and the result is obtained by taking absolute values, since $z\in B\iff |cz+d| \leq |c|r$ (or $<$,
depending on whether we are considering an open or closed ball.
\end{proof}

In view of the previous lemma, sets of the form $B(a,r)$ or $\PP^1 \setminus B(a,r^+)$
will be called \emph{open balls}, and we will qualify them of \emph{proper} if we
want to restrict to sets of the first type. Similarly, the sets $B(a,r^+)$
and $\PP^1\setminus B(a,r)$ will be called \emph{closed balls}.
Therefore, $\PGL_2(K)$ acts on the set of open (respectively closed) balls.

Given an open ball $B$, we will denote by $B^+$ the corresponding closed ball,
given by $B^+=B(a,r^+)$ if $B=B(a,r)$ and $B^+=\PP^1 \setminus B(a,r)$ if $B=\PP^1 \setminus B(a,r^+)$.

\begin{remark}
    If $g\in \PGL_2(K)$ and $g^{-1}(\infty) \notin \partial B$,
    then $g(\partial B) =\partial B(gP, |g'(P)|r)$.
    But if $g^{-1}(\infty) \in \partial B$, then
\[
  g(\partial B) =
  \PP^1\setminus \left(B\left( gP, |g'(P)|r\right) \cup B\left( g(\infty), \frac{|g'(\infty)|}{r}\right)  \right).
\]  
  Note that, in this last case,  $B(gP, |g'(P)|r)$ and $B(g\infty,\frac{|g'(\infty)|}{r})$
  are disjoint open balls with the same associated closed ball. 
\end{remark}

\subsection{Fundamental domains}

Let $\Gamma\subset \PGL_2(K)$ be a Schottky group,
with fixed generators $\gamma_1,\ldots,\gamma_g$.
We will write $\gamma_{-i}$ for $\gamma_i^{-1}$.
Denote by $\Sigma\subset \PP^1(K)$ the closure of the set of limit points for $\Gamma$, and set $\Omega=\PP^1(K)\setminus \Sigma$.

\begin{definition}
    A good fundamental domain for $\Gamma$ is a set of generators $\gamma_1,\ldots,\gamma_g$ and a connected affinoid set
    \[
      \cF^+=\PP^1 \setminus \bigcup_{i=\pm 1}^{\pm g} B_i,
    \]
    where the $B_i$ are open balls for $i=-g,\dots,1,1,\dots,g$ (we will abbreviate it as $i=\pm 1,\dots,\pm g$ from now on) such that:
    \begin{enumerate}
      \item The closed balls $B_i^+$ are mutually disjoint, and
      \item For all $i=\pm 1,\dots \pm g$ we have $\gamma_i(\PP^1\setminus B_{-i})=B_{i}^+$ for all $i$.
    \end{enumerate}
\end{definition}

As explained in \cite[Section 2]{MR}, there need not exist a good fundamental domain attached to a fixed set of 
generators. When a set of generators appears in a fundamental domain they are said to be in \emph{good position}.
Write $\cF=\PP^1\setminus\bigcup_{i=\pm 1}^{\pm g} B_i^+$. Gerritzen and van der Put proved the following:

\begin{theorem}[Gerritzen--Van der Put]
  Let $\Gamma\subset \PGL_2(K)$ be a Schottky group. Then there exists a good fundamental domain. Moreover, with the notations
  above, if $\cF^+$ is one such fundamental domain, then:
  \begin{itemize}
    \item If $\gamma\in\Gamma$ is not the identity, then $\gamma \cF^o \cap \cF = \emptyset$.
    \item If $\gamma \cF^+ \cap \cF^+ \neq \emptyset$, then $\gamma$ is either the identity or $\gamma=\gamma_i$ for some $i$.
    \item $\bigcup_{\gamma\in\Gamma} \gamma \cF^+ = \Omega$.
  \end{itemize}

  Conversely, let $B(P_i, \rho_i)$ (for $i=\pm 1,\dots,\pm g$) be a collection of open \emph{proper} balls with centers $P_i$ in $K$ and such that
  \begin{itemize}
    \item $\rho_i\rho_{-i}\in |K^\times| \forall i=1,\dots, g$, and
    \item The closed balls $B(P_i, \rho_i^+)$ are pairwise disjoint.
  \end{itemize}
  Then there exists $\gamma_1,\ldots, \gamma_g\in \PGL_2(K)$ generating a Schottky group and in good position with respect to
  \[
  \cF^+ = \PP^1 \setminus \bigcup_{i=\pm 1}^{\pm g} B(P_i, \rho_i).
  \]
\end{theorem}
\begin{proof}
  See \cite[I.4]{GvdP}.
\end{proof}

 The
main result that we will use from \cite{MR} is the following:

\begin{theorem}[Morrison-Ren]
  Let $\Gamma\subset \PGL_2(K)$ be the group generated by a given finite set of matrices.
  There is an algorithm that decides whether $\Gamma$ is Schottky and,
  if it is, outputs another set of generators in good position together with a fundamental domain.
\end{theorem}
\begin{proof}
  See \cite[Section 4]{MR}.
\end{proof}

\begin{remark}
  Suppose that $\Gamma\subset \PGL_2(K)$ is a Schottky group and that $\gamma_1,\ldots,\gamma_g$ are generators
  in good position. The fundamental domain $\cF^+$ need not be unique for such a set of generators,
  see Subsection~\ref{ssec:two-fundoms-same-gens} for a concrete example.

  Given a set of generators in good position, then it is obvious that any permutation of them is also
  in good position (the set of balls needs to be appropriately permuted).
  Moreover, if any of the generators is replaced by its inverse then one can simply exchange the two balls
  corresponding to this generator to obtain a fundamental domain for this new set. Are there other sets of
  generators in good position? In Subsection~\ref{ssec:two-gens-good-pos} we show that the answer is negative
  in general, by exhibiting a concrete example.
  
  It would be interesting to explain under which conditions the phenomena described above can occur.
\end{remark}

\subsection{A recursive definition of \texorpdfstring{$\Gamma_n$}{Gamma n}}

Let $\Gamma$ be a Schottky group, with good generators 
$\{\gamma_1,\ldots, \gamma_g\}$.

  The function $\length \colon \Gamma \to \ZZ$ assigns to an element $\gamma\in\Gamma$
  the length of the reduced word $w(\gamma)$ representing it.
  
  For $n\in \ZZ_{\geq 0}$, denote by
  \[
    \Gamma_{n} = \{\gamma\in \Gamma ~|~ \length(\gamma) = n\},\text{ and }
    \Gamma_{\geq n} = \bigcup_{i=n}^\infty \Gamma_n.
  \]
  
  Given $\gamma\in\Gamma_{\geq 1}$, define $t(\gamma)=\gamma_i$ (the \emph{tail} of $\gamma$)
  if $\length(\gamma\gamma_{-i}) < \length(\gamma)$ and $h(\gamma)=\gamma_j$
  (the \emph{head} of $\gamma$)
  if $\length(\gamma_{-j}\gamma) < \length(\gamma)$.
  In other words, if we can write $\gamma = \tilde \gamma \gamma_i$ with no cancellation,
  then $t(\gamma)=\gamma_i$ (and similarly for the \emph{head}).

  Finally, for $n\ge 1$ and $i\in \{\pm 1, \ldots \pm g\}$, denote by
  \[
  \Gamma_{n}^{(i)} =
  \{\gamma\in \Gamma_n ~|~  h(\gamma)=\gamma_i\}
  \]
  the subset of elements starting with the generator $\gamma_i$ and of length
  $n$.

\begin{lemma}
  \label{lem:decomposition}
  For all $n\ge 1$,
  \[
  \Gamma_{n+1}^{(i)} =
  \bigsqcup_{j\ne -i} \gamma_i\Gamma_{n}^{(j)}.
  \]
\end{lemma}

\begin{proof}
  All elements $\gamma$ in $\Gamma_{n+1}^{(i)}$
  are of the form $\gamma_i \tilde\gamma$ for a unique $i$ and $\tilde \gamma \in \Gamma_{n}$.
  Now $h(\tilde\gamma)\neq \gamma_{-i}$, and so $\tilde\gamma \in \Gamma_{n}^{(j)}$ for
  some $j\neq -i$.
\end{proof}

\section{Theta functions}
\label{sec:theta}
In this section we define the Theta function associated to a Schottky group $\Gamma$, as
a pairing on divisors of degree zero. This departs slightly from the point of view taken in
\cite{GvdP} or in \cite{MR}, but we found it easier to embrace extra flexibility of this approach.

\subsection{Definition}

Recall the pairing $\Div^0(\PP^1(K))\times \Div^0(\PP^1(K)) \to \PP^1(K)$
obtained by extending by linearity cross-ratio map
\[
  (z-w, a-b) \mapsto (z,w;a,b).
\]
Here, $(z,w;a,b)$ is the cross-ratio of these four points, which is defined
(when they are all distinct and different from $\infty$) as
\[
  (z,w;a,b)  = \frac{z-a}{z-b}\frac{w-b}{w-a}.
\]
This can be extended to all of $\PP^1(K)$ by imposing that
\[
  (g z, g w; g a, g b)=(z,w;a,b),\quad\forall g\in\PGL_2(K).
  \]
It can be extended to the case where two of the four points coincide, by
the formulas
\begin{align*}
  (z,w;z,b) &= (z,w;a,w) = 0,\\
  (z,z;a,b) &= (z,w;a,a) = 1,\\
  (z,w;a,z) &= (z,w;w,b) = \infty.\\
\end{align*}

The following proposition is well-known and justifies the choices made above:
\begin{proposition}
  The pairing
  \[
    (\cdot,\cdot)\colon \Div^0(\PP^1(K))\times \Div^0(\PP^1(K))\to \PP^1(K)
  \]
  satisfies:
  \begin{enumerate}
    \item For all $\gamma\in\PGL_2(K)$, $(\gamma D,\gamma E) = (D,E)$.
    \item For all $D$ and $E$, we have $(D,E) = (E,D)$.
    \item For all $D_1,D_2$ and all $E$,
      we have $(D_1+D_2,E) = (D_1,E)(D_2,E)$.
  \end{enumerate}
\end{proposition}

For any finite subset $\Sigma\subseteq \Gamma$, define
\[
  (D,E)_\Sigma =\prod_{\gamma\in\Sigma} (D, \gamma E),
\]
We will also denote by $(D,E)_n$ and $(D,E)_{\leq n}$ the quantities
$(D,E)_{\Gamma_n}$ and $(D,E)_{\Gamma_{\leq n}}$, respectively.

The goal of this article is to give an algorithm to efficiently compute
the \emph{Theta pairing} attached to $\Gamma$, which sends a pair of divisors
of degree zero $D$, $E$ to
\[
  (D,E)_\Gamma =\lim_{n\to\infty} (D,E)_{\leq n},
    \quad D, E \in \Div^0(\Omega).
\]

The existence of the limit above is established in~\cite{GvdP}, but in Section~\ref{ssec:convergence}
we will carefully study its rate of convergence.

\subsection{Analytic functions, the canonical embedding and uniformization of the Jacobian}

Since $(D,E)_\Gamma$ is $\Gamma$-invariant in both arguments, it can be seen
as a pairing on the group of coinvariant divisors
$\Div^0(\Omega)_\Gamma= H_0(\Gamma,\Div^0(\Omega))$. There is a canonical group
homomorphism
\[
\iota  \colon \ZZ[\Gamma] \to H_0(\Gamma,\Div^0(\Omega)),\quad \gamma \mapsto \gamma z_0 - z_0,
\]
where $z_0\in \Omega$ is any choice of base point. This is so because if $z_1$ is another choice, then
\[
(\gamma z_0 - z_0) - (\gamma z_1 - z_1) = (\gamma - 1)\cdot (z0 - z1) = 0
\in H_0(\Gamma, \Div^0(\Omega)).
\]
We obtain a group homomorphism
\[
u \mapsto \ZZ[\Gamma] \to \cO(\Omega)^\times,
  \quad \gamma \mapsto u_\gamma(z) = (z-\infty, \iota(\gamma))_\Gamma.
\]
The bi-additivity of $(\cdot, \cdot)_\Gamma$ implies that this map is trivial on the commutator
subgroup $[\Gamma,\Gamma]$ of $\Gamma$. We have
\begin{proposition}
  If $\gamma \notin [\Gamma, \Gamma]$, then $u_\gamma(z)$ is non-constant.
\end{proposition}
\begin{proof}
  See~\cite[page 59]{GvdP}.
\end{proof}

\newcommand{\dlog}{\operatorname{dlog}}

\begin{proposition}
  Suppose that $g\geq 3$, and let $\gamma_1,\ldots, \gamma_g$ be a set of free generators for $\Gamma$.
  The map
  \[
    \Omega/\Gamma \to \PP^{g-1}_K,
    \quad z \mapsto (\dlog u_{\gamma_1} : \dlog u_{\gamma_2} : \cdots : \dlog u_{\gamma_g})
  \]
  is the canonical embedding of the curve $\Omega/\Gamma$ in projective space.  
\end{proposition}
\begin{proof} 
  See~\cite[VI.4]{GvdP} and~\cite[\S 3.3]{MR}.
\end{proof}

Consider the bilinear pairing
\[
  \langle \cdot,\cdot\rangle \colon \Gamma^\ab \times \Gamma^\ab \to K^\times,
  \quad (\alpha, \beta) \mapsto (\iota(\alpha), \iota(\beta))_\Gamma.
\]
Its matrix is the \emph{period matrix} of the Jacobian $\operatorname{Jac}(\Omega/\Gamma)$:

\begin{theorem}
  The Jacobian $\operatorname{Jac}(\Omega/\Gamma)$ is isomorphic to $(K^\times)^g / \Lambda$,
  where $\Lambda$ is the $g \times g$ period matrix defined above.
\end{theorem}
\begin{proof}
  See~\cite[VI.2]{GvdP}.
\end{proof}

\subsection{Theta functions of discontinuous groups}

We start by proving a simple lemma.
\begin{lemma}
  \label{lem:theta-conjugate}
  For all $g \in \PGL_2(K)$ we have
  \[
  (gD,gE)_\Gamma = (D,E)_{g^{-1}\Gamma g}.
  \]
  In particular, if  $g$ belongs to the normalizer of $\Gamma$ in $\PGL_2(K)$, then
  \[
    (gD,gE)_\Gamma = (D,E)_{\Gamma}.
  \]
\end{lemma}
\begin{proof}
  It is a matter of a simple computation:
  \[
  \prod_{\gamma \in \Gamma} (gD,\gamma g E) = \prod_{\gamma \in \Gamma} (D,g^{-1} \gamma g E).
  \]
  Since $\gamma \mapsto g^{-1}\gamma g$ is a bijection $\Gamma \to g^{-1}\Gamma g$ the result follows.
\end{proof}

The above result allows for the definition (already present in~\cite{GvdP}) of the Theta function
attached to any discontinuous group $G$. Indeed, let $\Gamma \trianglelefteq G$ be a finite-index
normal Schottky subgroup, and write
\[
G = \bigcup_{i=1}^h \Gamma g_i.
\]
Define the Theta function attached to $G$ as
\[
(D,E)_G = \prod_{i=1}^h (D,g_iE)_\Gamma = (D, \sum_{i=1}^h g_i E)_\Gamma.
\]
\begin{proposition}
  The definition of $(\cdot, \cdot)_G$ does not depend on the choice of $\Gamma$.
\end{proposition}
\begin{proof}
  By taking the intersection of any two given normal finite-index Schottky groups,
  we reduce to showing that, if $\Gamma_2 \trianglelefteq \Gamma_1$ are two Schottky groups with
  $[\Gamma_1:\Gamma_2]<\infty$, then $(D,E)_{\Gamma_1} = \prod_{j=1}^s (D,g_jE)_{\Gamma_2}$, where
  $\{g_1,\ldots,g_s\}$ are coset representatives for the quotient. Since each element
  $\gamma \in \Gamma_1$ can be uniquely written as $\tau g_j$ for a unique $\tau\in \Gamma_2$
  and $j\in\{1,\ldots s\}$, we have
  \[
  (D,E)_{\Gamma_1} = \prod_{\gamma \in \Gamma_1} (D,\gamma E) = \prod_{j=1}^s\prod_{\tau \in \Gamma_2} (D, \tau g_j E) =
  \prod_{j=1}^s (D, g_jE)_{\Gamma_2}.
  \]
\end{proof}

Thanks to the previous results, we can reduce the computation of the Theta function
of a non-archimedean discontinuous group to that of a Schottky group, at essentially no extra
cost, since as we will see the performance of the algorithm very insensitive to the input
divisor $E$. This assumes that we are given a finite index normal Schottky group, which
in some cases of interest (for example $p$-adic Whittaker groups, or quaternionic groups) is
feasible. See~\cite{amoros-milione} or~\cite[Chapter 9]{GvdP} for more information.

\subsection{Convergence properties}
\label{ssec:convergence}

Since our goal is to approximate $(D,E)_\Gamma$ with $(D,E)_{\Gamma_{\leq n}}$,
we need to estimate
for which $n$ our approximation is good enough. We assume henceforth that
$\infty \in \Omega$. If $\Omega \cap \PP^1(K) \neq \emptyset$ this can be accomplished
by conjugating $\Gamma$ and using Lemma~\ref{lem:theta-conjugate}. Otherwise, one may
replace $K$ by any larger field and reduce to the previous situation.

To each $\gamma\in\Gamma_{\geq 1}$ we attach a corresponding open ball $B(\gamma)$,
by setting
\[
  B(\gamma)= \gamma(\PP^1\setminus B_{-i}^+),
  \quad \text{if } t(\gamma)=\gamma_i.
\]
Note that if $\length(\gamma_1\gamma) = \length(\gamma_1)+\length(\gamma)$ then
\[
B(\gamma_1\gamma) = \gamma_1 B(\gamma).  
\]

\begin{proposition}
  \label{prop:infty-not-in-balls}
  There exists $n=n(\Gamma)\geq 1$ such that
  for all $\gamma\in \Gamma_{\geq n}$ we have $B(\gamma)^+$ is proper (i.e. $\infty \notin B(\gamma)^+$).
\end{proposition}
\begin{proof}
  Suppose that for all $n$ there is some $\gamma_n\in \Gamma_{\geq n}$
  with $\infty \in B(\gamma)^+$.
  One can then extract a sequence $(\gamma_{m})_{m\geq 1}$ of pairwise
  distinct elements
  such that $\infty \in B(\gamma_m)$ for all $m$. For any $z \in \cF$,
  we have $\gamma_m z \in B(\gamma_m)$. But then $|\gamma_m z|\to \infty$,
  hence $\gamma_m z \to \infty$ with $m$, which contradicts the
  assumption that $\infty\in\Omega$.
\end{proof}

\begin{remark}
  The $n(\Gamma)$ in the above proposition can be algorithmically computed:
  it suffices to calculate $B(\gamma)$ for all $\gamma \in \Gamma_n$ for increasing $n$.
  Once we reach some $n$ for which none of the $B(\gamma)$ contain $\infty$,
  we know this condition will be preserved for all $m\geq n$.
\end{remark}

When the balls $B(\gamma)$ are proper, we would like to estimate their radii $r(\gamma)$
in terms of the length of $\gamma$. Let $\gamma, \gamma' \in \Gamma_{\geq n}$. Then
$B(\gamma')\subseteq B(\gamma)$ if and only if $\gamma$ is a prefix of $\gamma'$
(that is, $\gamma'=\gamma t$ with $\length(\gamma')=\length(\gamma)+\length(t)$).
Also, $\gamma(\infty) \in B(\gamma)$ for all $\gamma\in \Gamma_{\geq 1}$, and
\[
  \PP^1 \setminus \bigcup_{\gamma \in \Gamma_{<n}} \gamma \cF 
  = \bigcup_{\gamma\in\Gamma_n} B(\gamma).
\]
These facts can be found in~\cite[Section 4.1]{GvdP}.

For $n=n(\Gamma)$ as in the proposition above, write
\[
  \rho = \max\left\{\frac{r(\gamma_i\gamma)}{r(\gamma)} ~|~
  i \in \{\pm 1,\ldots, \pm g\}, \gamma \in \Gamma_n^{\neq -i}\right\}.
\]
Note that $\rho<1$, since we are considering finitely many terms,
each of which is less than $1$ because
$B(\gamma_i\gamma)\subsetneq B(\gamma)$. Now set the constant $C$, depending on $\Gamma$ and the generators, as 
\[
  C =\frac{ \max_{\gamma\in \Gamma_n}  r(\gamma)}{\rho^n}.
\]
We have the following result.
\begin{proposition}
Let $n = n(\Gamma)$ as above. Then for all $\gamma\in \Gamma_{\geq n}$,
\[
  r(\gamma) \leq C \rho^{\length(\gamma)}.
\]
\end{proposition}
\begin{proof}
  By induction on the length $m\geq n$ of $\gamma$. For $m=n$, the result holds
  trivially by definition of $C$. Now, in general consider $\gamma_i t g$,
  a proper word decomposition, where $g\in \Gamma_n$. We will show that
  \[
    \frac{r(\gamma_i t g)}{r(\gamma_i g)} = \frac{r(t g)}{r(g)} \leq  \rho^{\length(t)}.
  \]
  Therefore by induction hypothesis we have
  \[
    r(\gamma_itg) \leq  r(\gamma_ig)\rho^{\length(t)} \leq
    C \rho^{\length(\gamma_ig)+\length(t)}
    \]
    as desired. To show that
    \[
  \frac{r(\gamma_i t g)}{r(\gamma_i g)} = \frac{r(t g)}{r(g)}\leq  \rho^{\length(t)},
    \]
    observe that the equality holds by Lemma~\ref{lemma:action-balls}. We show the
    inequality by induction on $\length(t)$, the base case being trivial by definition. In general, write
    $t=\gamma_i t'$. Then
    \[
      \frac{r(\gamma_i t' g)}{r(g)} = \frac{r(\gamma_i t' g)}{r(\gamma_i g)} \frac{r(\gamma_i g)}{r(g)} \leq \rho^{\length(t')} \cdot \rho
    \]
    by induction hypothesis and the definition of $\rho$.
\end{proof}

Set $\delta(\infty)=1$ and, for any point $z \in \cF^+ \setminus\{\infty\}$, define $\delta(z)$ to be
\[
  \delta(z) = \max \{|z - c(\gamma)|^{-1}  ~|~ \gamma\in \Gamma_n \} > 0,
\]
where $c(\gamma)$ denotes a chosen center of $B(\gamma)$. Also, if $D\in \Div^0(\cF^+)$,
define
\[
  \delta(D) = \max \{\delta(z) ~|~ z \in |D|\} >0.
\]
Finally, let
\[
R_D = \max \{|z-w| ~|~ z,w \in |D|\setminus \{\infty\}\} 
\]
that we call the diameter of $D$. 

\begin{proposition}
  Let $D, E\in\Div^0(\cF^+)$, and let $0<\eps<1$ be given.
  Define
  \[
    \nu=\nu(\eps) = \max\left\{2, n(\Gamma), \frac{\log \eps - \log C - 2 \log \delta(D) - R_D}{\log \rho}\right\}.
  \]
  Then for all $\gamma\in\Gamma_{> \nu}$ we have
  \[
    |(D,\gamma E) -1| < \eps
  \]
\end{proposition}
\begin{proof}
We will show  that \[    |(D,\gamma E) - 1| \leq C R_D \delta(D)^2 \rho^{\length(\gamma)}\le \eps\]
for any $D, E\in\Div^0(\cF^+)$ and $\gamma\in\Gamma_{> \nu}$, for $\nu\ge n:=n(\Gamma)$, $\nu\ge 2$, and the result will follow.

  Firstly, we show the result when $E$ and $D$ are elementary divisors. By the symmetry properties of the cross-ratio, we have
  \[
    |(z,w;\gamma a, \gamma b) - 1| = |(z,\gamma a, w, \gamma b)| =
    \frac{|z-w|}{|z-\gamma b|}\frac{|\gamma a - \gamma b|}{|w - \gamma a|}.
  \]
  
  If $z$ or $w$ is $\infty$, then by definition of the cross-ratio we need instead
  to bound
  \[
    \frac{|\gamma a - \gamma b|}{|w-\gamma a|} \text{ or }
    \frac{|\gamma a - \gamma b|}{|z-\gamma b|}
  \]
  respectively. We will separately bound the quantities $|\gamma a - \gamma b|$
  and $|w-\gamma a|$. Note also that, since both $\gamma a$ and $\gamma b$ belong to $B(\gamma)$, we have
  \[
    |\gamma a - \gamma b| < r(\gamma).
  \]
  
  On the other hand, since $w \in \cF^+$, it is not in $B(\gamma)^+$
  for any $\gamma \in \Gamma_{\geq 2}$. Therefore, if $\gamma \in \Gamma_{\geq n}$,
  as long as $n\geq 2$, we have
  \[
    |w - \gamma a| = |w - c(\gamma)| \geq  \delta(w)^{-1}.
  \]
  We obtain
  \[
    |(z, w;\gamma a, \gamma b) -1| \leq r(\gamma) |z - w| \delta(z)\delta(w),
  \]
  where the term $|z-w|$ must be omitted if $z=\infty$ or $z=\infty$.

  Using the bound on $r(\gamma)$, we have, for all $\gamma\in \Gamma_{\geq n}$,
  \[
    |(z,w;\gamma a,\gamma b) - 1| \leq C |z-w| \delta(z)\delta(w)\rho^{\length(\gamma)},
  \]
and the result follows in this first case. 

The general case is done decomposing the divisors $D$ and $E$ as sums of degree 0 elementary divisors, and using that, 
for any $\alpha$, $\beta \in K$, if $|\alpha-1|<\eps$ and $|\beta-1|<\eps$, then $|\alpha\beta-1|<\eps$. 
\end{proof}
\begin{corollary}
  \label{cor:approximation}
  Let $D,E\in\Div^0(\cF^+)$ and let $0<\eps<1$ be given.
  Then for all $n\geq \nu(\eps)$ we have
  \[
    |(D,E)_\Gamma/(D,E)_{\leq n} - 1|< \eps.
  \]
\end{corollary}

\section{The main algorithm}
\label{sec:algorithm}
In this section, we describe the proposed algorithm to compute arbitrary
approximations to $(D,E)_\Gamma$,
for divisors $D, E \in \Div^0(\Omega)$.
We will assume that a fundamental domain $\cF$
has been computed for $\Gamma$, and that the corresponding balls have center in $K$ and radius
in $|K^\times|$.

\begin{lemma}
  The map $\Div^0(\cF^+) \to \Div^0(\Omega)_\Gamma$ is surjective.
\end{lemma}
\begin{proof}
  It is enough to show that every divisor of the form $(a)-(b)$ is $\Gamma$-equivalent
  to a divisor supported on $\cF^+$. For each $i=1,\ldots,g$, fix a point $z_i\in\cF^+$
  such that $z_i' = \gamma_i z_i$ is also in $\cF^+$. This can be done by taking $z_i$
  any point in $\partial B_{-i}$, so that $\gamma_i z_i \in \partial B_i$.
  
  Translating by $\Gamma$, we may and do assume that
  $b \in \cF^+$. Let $a_0\in \cF^+$ be $\Gamma$-equivalent to $a$, say
  $a = \alpha a_0$ for some $\alpha\in\Gamma$. We do induction on $\length(\alpha)\geq 0$.
  If $\alpha=1$ we are done, otherwise
  write $\alpha=\gamma_i \tilde \alpha$, with $\gamma_i$ one of the generators of $\Gamma$.
  If we set $\tilde a = \tilde \alpha a_0$, we can decompose
  \begin{align*}
    (a)-(b) &= (a) - (z_i') + (z_i') - (b)\\
    &= (\gamma_i \tilde\alpha a_0) - (\gamma_i z_i) + (z_i') - (b)\\
    &= \gamma_i \left((\tilde a) - (z_i)\right) + (z_i') - (b).
  \end{align*}
  The first term is $\Gamma$-equivalent to $(\tilde a)-(z_i)$ which is,
  by induction, in $\Div^0(\cF^+)$. The second term is also in $\Div^0(\cF^+)$
  by our choice of $z_i$ and $z_i'$, thus proving the claim.
\end{proof}

Note that the proof of the previous lemma is constructive, so we may and do assume
that both divisors $D$ and $E$ are supported on $\cF^+$.
As we have seen in Corollary~\ref{cor:approximation}, in order to approximate
$(D,E)_\Gamma$ it is enough to compute the finite product
\[
  (D,E)_{\leq \nu} = \prod_{k=0}^{\nu} (D,E)_k
\]
for an appropriate computable $\nu=\nu(\eps)$.

In the rest of the section we describe how to iteratively compute $(D,E)_k$.

\subsection{An iterative algorithm}

The basis for the algorithm rests on Lemma~\ref{lem:decomposition}. Note that
for $k\geq 1$ we have
\[
  (D,E)_{\Gamma_k} = \prod_{i=\pm 1}^{\pm g} (D,E)_{\Gamma_k^{(i)}}.
\]
Moreover, Lemma~\ref{lem:decomposition} implies that for $k\geq 1$ we have
\begin{align*}
  (D,E)_{\Gamma_{k+1}^{(i)}} &= \prod_{\gamma\in \Gamma_{k+1}^{(i)}} (D,\gamma E)
    =\prod_{j\neq i} \prod_{\gamma \in \gamma_i \Gamma_k^{(j)}} (D, \gamma E)
    =\prod_{j\neq i} \prod_{\gamma \in \Gamma_k^{(j)}} (D, \gamma_i\gamma E)\\
  &=\prod_{j\neq i} \prod_{\gamma \in \Gamma_k^{(j)}} (\gamma_i^{-1}D, \gamma E)
  =\prod_{j\neq i} (\gamma_i^{-1}D, E)_{\Gamma_k^{(j)}} 
\end{align*}

Suppose that we have computed a $2g$-tuple of rational functions
$\Phi_k = (\phi_k^{(i)})_i\in (K(t)^\times)^{2g}$ with 
\[
  \div \phi_k^{(i)} = \sum_{\gamma \in \Gamma_k^{i}} \gamma E.
\]
Then the tuple $\Phi_{k+1} =  (\phi_{k+1}^{(i)})_i$ could be computed
by setting
\[
  \phi_{k+1}^{(i)} = \prod_{j\neq -i} \gamma_i \phi_{k}^{(j)},
\]
where the action of $\PGL_2(K)$ on $K(t)^\times$ is
\[
  (\gamma \phi)(t) = \phi(\gamma^{-1} t).
\]

The iterative algorithm that these observations would lead to
has a  serious problem. Namely, that the rational
functions $\phi_{k}^{(i)}(t)$ have degree that grows exponentially with $k$.
This results in an algorithm which is exponential both in the time and in
space. In the next subsection we describe how to turn it into a polynomial
time algorithm.

\subsection{Affinoid domains and computation}

We consider the groups of units in the affinoid algebras $\cO(B_i^c)^\times$.
The group $K(B_i^c)^\times$ of
rational functions with zeros and poles contained in $B_i$ embeds
naturally in $\cO(B_i^c)^\times$.

\begin{lemma}
  For each $i, j$ with $j\neq -i$ the map $\phi(t) \mapsto \phi(\gamma^{-1} t)$
  gives a map
  \[
    f_{i,j} \colon \cO(B_j^c) \to \cO(B_i^c),
    \quad \phi \mapsto \gamma_i \phi
  \]
  which fits into a commutative diagram
  \[
    \xymatrix{ 
      \cO(B_j^c)^\times \ar[r]^{f_{i,j}} & \cO(B_i^c)^\times \\
      K(B_j^c)^\times \ar[u]\ar[r]^{\phi \mapsto \gamma_i \phi} &
        K(B_i^c)^\times\ar[u]
    }
  \]
\end{lemma}
\begin{proof}
  The affinoid ring $\cO(B_j^c)$ consists of rigid meromorphic functions
  on $\PP^1(K)$ having poles inside $B_j$ (see for instance~\cite[Section 2.2]{FvdP}).
  
  Recall that if $j\neq -i$ then $\gamma_i(B_j)\subseteq B_i$.
  Therefore, if $\phi$ is a meromorphic function having its negative part of the
  divisor of zeros and poles $\div^-\phi$ satisfying
  $|\div^- \phi| \subseteq B_j$, then
  \[
    |\div^- \gamma_i \phi| = \gamma_i (|\div^- \phi|) \subseteq  B_i.
  \]
\end{proof}

Denote by $\cR$ the product
\[
  \cR = \prod_{i=\pm 1}^{\pm g} \cO(B_i^c)^\times
\]
and by $\cR_0$ the product
\[
  \cR_0 = \prod_{i=\pm 1}^{\pm g} K(B_i^c)^\times,
\]
which embeds in $\cR$.
The maps $f_{i,j}$ can be put together into an operator $\nabla$ on $\cR$:
\[
  \nabla \colon \cR \to \cR,\quad
  \vec F = (F_i)_i \mapsto (\nabla_i \vec F)_i,
\]
with
\[
\nabla_i \vec F = \prod_{j \neq i} f_{i,j}(F_j).
\]
We write $\nabla_0$ for the analogous map on $\cR_0$.

\begin{proposition}
  We have the following commutative diagram
  \[
    \xymatrix{
      \cR \ar[r]^{\nabla} & \cR\\
      \cR_0\ar[u]\ar[r]^{\nabla_0}&\cR_0\ar[u]
    }
  \]
\end{proposition}
\begin{proof}
  Obvious from the definitions.
\end{proof}

Recall that on an affinoid algebra $\cO(F)$, the supremum norm is defined as $\| \varphi\| = \sup_{z \in F} |\varphi(z)|$.
The following result is the key we need for the algorithm to work

\begin{theorem}
  For each $j\neq -i$ the map $f_{i,j}\colon \cO(B_j^c) \to \cO(B_i^c)$ satisfies
  \[
    \| f_{i,j}(\varphi) \| \leq \| \varphi \|, \quad \forall \varphi \in \cO(B_j^c).
  \]
\end{theorem}
\begin{proof}
  Let $\varphi \in \cO(B_j^c)$. We need to show that
\[
  |\varphi(\gamma_i^{-1} x)| \leq \| \varphi \|,\quad \forall x \in B_i^c.
\]
Now, if $x \in B_i^c$,
then it follows that $\gamma_i^{-1}(x) \in B_{-i}$. In particular, since $j\neq -i$ then
$\gamma_i^{-1}x \in B_j^c$.
Therefore,
\[
|\varphi(\gamma_i^{-1} x)| \leq \sup_{z \in B_j^c} |\varphi(z)|=\|\varphi\|,
\]
as wanted.
\end{proof}

We want to work with approximations to $\varphi$, and the following result
says that the action of $f_{i,j}$ can be computed on such approximations.

\begin{corollary}
  Suppose that $\|\varphi\|\leq 1$, and let $\tilde\varphi = \varphi\cdot (1+e)$,
  with $\|e\| \leq \eps$. Then
  \[
  \|f_{i,j} \tilde \varphi - f_{i,j} \varphi\| \leq \eps.
  \]
\end{corollary}
\begin{proof}
  Since $\tilde\varphi - \varphi = \varphi e$, we have
  \[
    \|\tilde \varphi - \varphi\| \leq \|\varphi\| \|e\| \leq \eps.
  \]
  The previous result proves the claim.
\end{proof}

In a similar vein, we have
\begin{corollary}
  Suppose that $\vec F \in \cR$ has $\|\vec F\|\leq 1$, and let $\vec G = \vec F (1 + e)$,
  with $\|e\| \leq \eps$. Then
  \[
    \|\nabla \vec F - \nabla \vec G\| \leq \eps.
  \]
\end{corollary}

\subsection{Description of the algorithm}

We have now all ingredients to describe the algorithm. Suppose we have already
computed a fundamental domain
\[
  \cF^+=\PP^1\setminus \bigcup_{i=\pm 1}^{\pm g} B_i.
\]
Suppose given divisors $D, E \in \Div^0(\Omega)$, and $\eps > 0$. The goal is to
compute $\theta$ approximating the theta function $(D,E)_\Gamma$: that is,
we want to find $\theta$ such that
\[
  \left|(D,E)_\Gamma / \theta - 1\right| < \eps.
\]

We have seen that finding $\theta$ amounts to computing $(D,E)_{\leq \nu}$
for an explicit $\nu=\nu(\eps)$.

Next, using $\Gamma$-invariance, we can assume that $D$ and $E$ have support on $\cF^+$.
Using additivity, we can even assume that $D=(z_0)-(\infty)$ and $E=(a)- (b)$.

We will compute rational functions $\varphi_0(t), \varphi_1(t)\in K(t)$, and elements
$\vec F_2, \ldots,\vec F_{\nu} \in \cR$ as follows:
\[
\varphi_0(t) = \frac{t - a}{t-b},
\quad \varphi_1(t) = \prod_{i=\pm 1}^{\pm g} \frac{t - \gamma_i^{-1} a}{t - \gamma_i^{-1} b}.
\]
For $\vec F_2 = (F_2^{(i)})_{i}$, set
\[
  F_2^{(i)} = \prod_{j \neq i} \varphi_1(\gamma_i^{-1} z) \in \cO(B_i^c)^\times.
\]
This element is computed approximately and normalized to having norm one.

Choose elements $\varpi_i\in \partial B(p_i,\rho_i)$, and denote by $\tau_i$ the matrix
taking $p_{-i} \mapsto 0$, $\varpi_i \mapsto 1$ and $p_i \mapsto \infty$.
\begin{lemma}
  We have
  \[
    \tau_i(\PP^1 \setminus B_i) = B(0,1^+).
  \]
\end{lemma}
\begin{proof}
  The matrix $\tau_i$ maps $p_{-i}$ to $0$ and $p_{i}$ to $\infty$. Since $\varpi_i$
  is mapped to $1 \in \partial B(0,1)$, it maps the boundary of $B_i$ to that of $B(0,1)$,
  as claimed.
\end{proof}

Using the isomorphisms $\tau_i \colon B_i^c \to B(0,1^+)$,
the function $F_2^{(i)}$ can be represented
as a power series $G_2^{(i)}(t)$, normalized so that its constant term is $1$.

Next, compute iteratively $\vec F_3,\ldots,\vec F_\nu$, by using
\[
  \vec F_{k+1} = \nabla \vec F_k.
\]

In terms of the power series representation, given $G_k^{(i)}$ for all $i$, we compute
\[
  G_{k+1}^{(i)}(t) = \prod_{j\neq i} G_k^{(j)}(\tau_j \gamma_i^{-1} \tau_i^{-1} t).
\]

\begin{lemma}
  For all $t \in B(0,1^+)$ we have
\[
  \tau_j \gamma_i^{-1}\tau_i^{-1} t\in B(0,1^+).
\]
\end{lemma}
\begin{proof}
  Since $\tau_i\gamma_i\tau_j^{-1} B(0,1^+)^c \subseteq B(0,1^+)^c$, taking complements
  we get $B(0,1^+) \subseteq \tau_i\gamma_i\tau_j^{-1} B(0,1^+)$. Therefore,
  if $t \in B(0,1^+)$ then it can be written as $t=\tau_i\gamma_i\tau_j^{-1} s$ for
  some $s\in B(0,1^+)$. Hence,
  \[
    |\tau_j \gamma_i^{-1}\tau_i^{-1} t| = |s| \leq 1.
  \]
\end{proof}

Finally, in order to evaluate the sought approximation to $((z_0) - (\infty), (a) - (b))_{\Gamma}$,
it is enough to compute
\[
  \varphi_0(z_0)\cdot \varphi_1(z_0) \cdot
  \prod_{k=2}^{\nu} \prod_{i=\pm 1}^{\pm g} F_k^{(i)}(z_0).
\]
In this last equation, note that $F_k^{(i)}(z_0) = G_k^{(i)}(\tau_i z_0)$.

\begin{remark}
  If several evaluations are to be done, it is worth computing the products
  \[
    F_{\leq \nu}^{(i)} = \prod_{k=2}^{\nu} F_k^{(i)} \in \cO(B_i^c),
  \]
  as well as
  \[
    \varphi_{\leq 1} = \varphi_0 \varphi_1 \in K(t)
  \]
  so that evaluating $((z_0) -(\infty), (a) - (b))_{\Gamma}$ just amounts to
  computing
  \[
  \varphi_{\leq 1}(z_0) \prod_{i=\pm 1}^{\pm g} F_{\leq \nu}^{(i)}(z_0).
  \]
\end{remark}

\section{Examples}
\label{sec:examples}
We have implemented all the above algorithms in Sage (\cite{sagemath}). They are
distributed as part of the \texttt{darmonpoints} package, available at
\url{github.com/mmasdeu/darmonpoints}. In this section we provide some examples
that illustrate the use of these algorithms.

\subsection{An example with two sets of generators in good position}
\label{ssec:two-gens-good-pos}
Consider the group $\Gamma=\langle \gamma_1,\gamma_2\rangle\subseteq \PGL_2(\mathbb{Q}_3)$,
where $\gamma_1$ and $\gamma_2$ are defined as
\[
\left(\begin{array}{rr}
4939\cdot 3^{-2} & 3244 \\
6560\cdot 3^{-8} & 2\cdot 3^{-6}
\end{array}\right),\quad
\left(\begin{array}{rr}
241 & 4 \\
80\cdot 3^{-6} & 2\cdot 3^{-6}
\end{array}\right)
\]
The group $\Gamma$ is a Schottky group, and the given generators are in good position.
The corresponding balls are:
\begin{align*}
  B_1 &= B(2\cdot 3^6 + 2\cdot 3^{10} + 2\cdot 3^{12},3^{-12})&
 B_{-1} &= B(2\cdot 3^2 ,3^{-4})\\
 B_2 &= B(2\cdot 3^6,3^{-11})&
 B_{-2} &= \PP^1 \setminus B(0,3^{-1})^+
\end{align*}

However, the set $\{\gamma_1' = \gamma_1^{-1}, \gamma_2'=\gamma_2\gamma_1^{-1}\}$
also generates $\Gamma$, and they form as well a set of generators in good position,
this time with respect to the balls
\begin{align*}
  B_1 &= \PP^1 \setminus B(2\cdot 3^6,p^{-6})^+&
 B_{-1} &= B(2\cdot 3^6 + 2\cdot 3^{10} + 2\cdot 3^{12} + 2\cdot 3^{18},p^{-18})\\
 B_2 &= B(2\cdot 3^6,p^{-11})&
 B_{-2} &= B(2\cdot 3^6 + 2\cdot 3^{10} + 2\cdot 3^{12} + 3^{14},p^{-15})
\end{align*}

This example shows that for a given group $\Gamma$ there might be one more set
of generators in good position,
even up to permutation and up to changing individual generators by their inverses.

\subsection{An example with two different fundamental domains for the same set of
generators}
\label{ssec:two-fundoms-same-gens}

Consider the balls in $\PP^1(\QQ_5)$
\[
  B_1 = B(2,5^{-1}), \quad B_{-1} = B(3,5^{-1}),
\]
and the matrix
\[
  \gamma_1=\begin{pmatrix}73&-144\\24&-47\end{pmatrix},
  \]
which satisfies $\gamma_1(\PP^1\setminus B_{-1}) = B_{1}^+$. Set also
\[
  \gamma_2 = \begin{pmatrix}5^6&0\\0&1\end{pmatrix}.
  \]
Define
\begin{align*}
  B_2 &= B(0,5^{-3}) & B_{-2} &= \PP^1\setminus B(0,5^{3})\\
  C_2 &= B(0,5^{-2}) & C_{-2} &= \PP^1\setminus B(0,5^{4}).
\end{align*}
Then the two sets of balls
\[
  \{ B_{\pm 1}, B_{\pm 2}\},\quad\{B_{\pm 1}, C_{\pm 2}\}
  \]
  are both fundamental domains for the pair $(\gamma_1,\gamma_2)$, as
  is easily checked.

\subsection{Performance of the algorithm}

Consider the group $\Gamma \subset \PGL_2(\mathbb{Q}_3)$ generated by the matrices
\[
  \gamma_1 = \begin{pmatrix}-5&32\\-8&35\end{pmatrix},
  \quad \gamma_2 = \begin{pmatrix}-13&80\\-8&43\end{pmatrix}.
\]
As verified in~\cite[Example 2.3]{MR}, these are in good position. A good fundamental domain
is given by
\[
B_1=B(4,1/9), B_{-1} = B(1,1/9), B_2=B(5,1/9), B_{-2} = B(2,1/9).  
\]

In Figure~\ref{fig:oc-vs-naive} we compare the iterative algorithm to the naive
evaluation, and in Figure~\ref{fig:timings-oc} we see how the iterative algorithms
performs as the precision increases below what is possible with the naive approach.

\begin{figure}
\label{fig:oc-vs-naive}
\includegraphics[width=.7\textwidth]{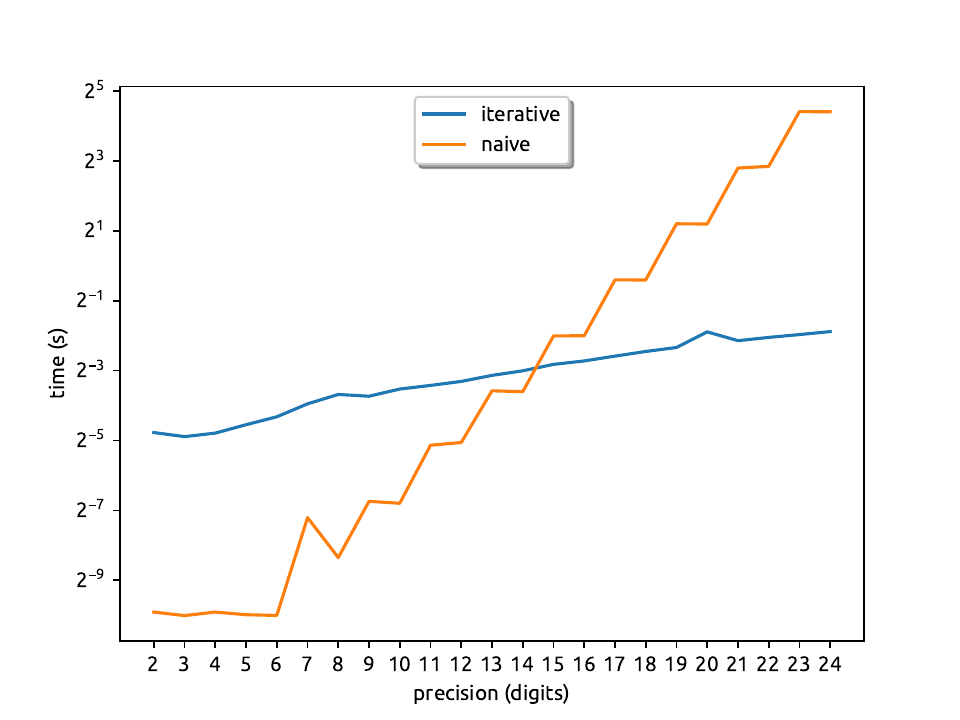}
\caption{Time comparison of naive versus iterative algorithms}
\end{figure}

In Figure~\ref{fig:timings-oc} we show how the iterative algorithm performs
for precisions much larger to what would be feasible with the naive algorithm. In
this plot the time is plotted on a linear scale, and we observe a quadratic dependency
on the precision, which matches the expectation from the design of the algorithm.
\begin{figure}
  \label{fig:timings-oc}
  \includegraphics[width=.7\textwidth]{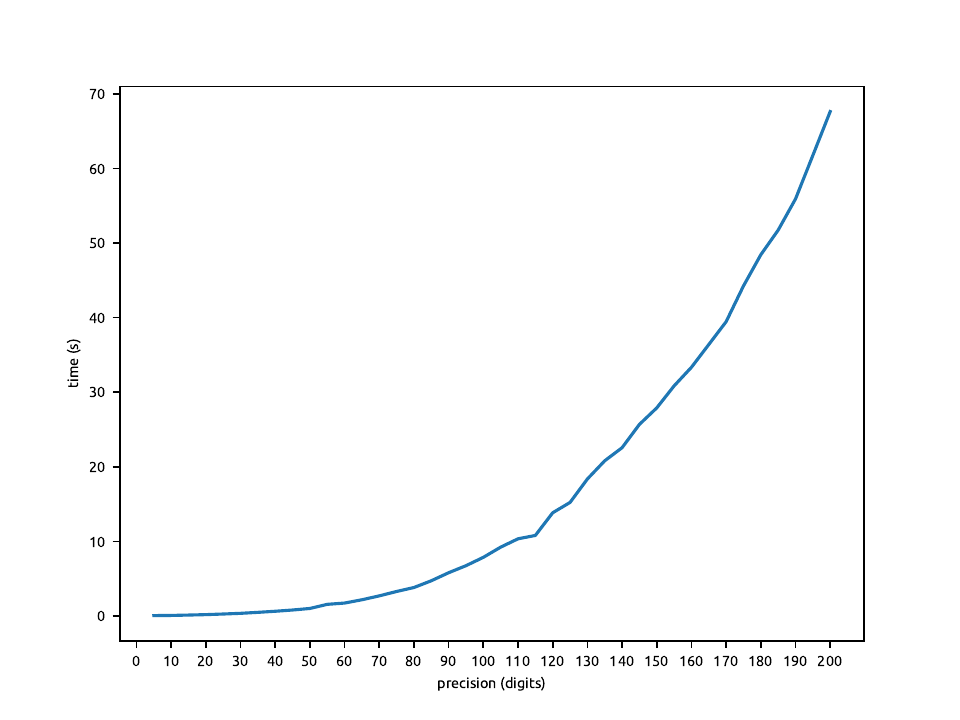}
  \caption{Performance of iterative algorithm for large precision}
  \end{figure}

In the next example, consider the group $\Gamma \subset \PGL_2(K)$, where
$K$ is a quadratic Eisenstein extension of $\mathbb{Q}_3$. In order to compute
the fundamental domain one needs to extend further to a quadratic ramified extension $L/K$. In
Figure~\ref{fig:oc-vs-naive-eisenstein} we that the naive method performs better
until a precision of $15$ $\pi$-adic digits is required, where $\pi$ is the uniformizer of $L$.
\begin{figure}
  \label{fig:oc-vs-naive-eisenstein}
  \includegraphics[width=.7\textwidth]{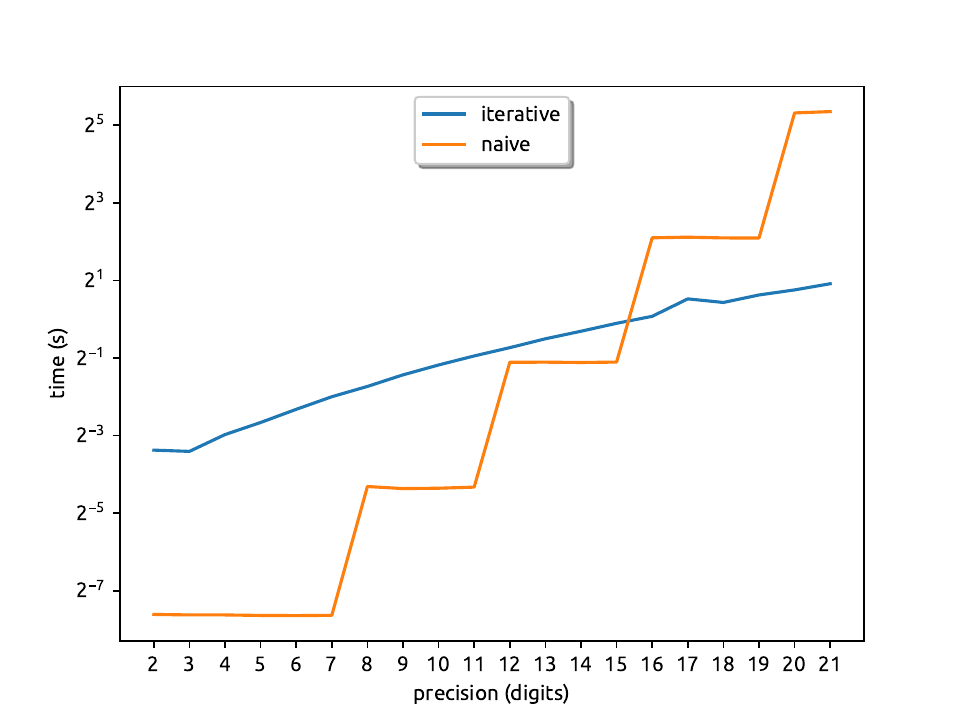}
  \caption{Time comparison for Schottky group defined over an Eisenstein extension}
  \end{figure}

Finally, we study the performance as the number of generators increases. We have
considered Schottky groups generated by $g=2$, $g=3$ and $g=4$ matrices with coefficients
in the $11$-adic field. In Figure~\ref{fig:oc-increasing-gens} we can observe how the
iterative algorithm outperforms the naive algorithm earlier as the number of generators
increases.

\begin{figure}
  \label{fig:oc-increasing-gens}
  \includegraphics[width=.7\textwidth]{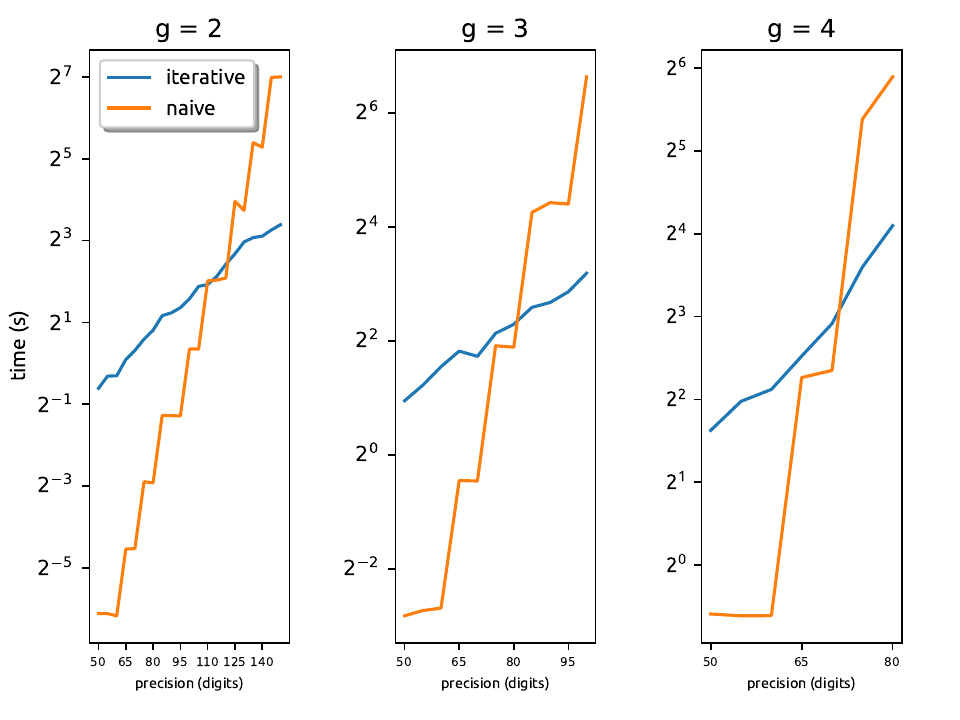}
  \caption{Time comparison for Schottky group on $2$, $3$ and $4$ generators}
  \end{figure}

\printbibliography

\end{document}

%% file: preamble.tex
\usepackage{amsmath, amsthm, amssymb}
\usepackage[T1]{fontenc}
\usepackage[backend=biber]{biblatex}
\usepackage[left=3cm, right=3cm, top=3cm, bottom=3cm]{geometry}
\usepackage{enumerate}
\usepackage{array}
\usepackage{graphicx}
\usepackage{hyperref}
\usepackage[dvipsnames,svgnames]{xcolor}
\usepackage{luatex85}
\usepackage[all, cmtip]{xy}
\usepackage{algorithm2e}
\usepackage{pdfpages}
\usepackage{svg}
\SetKwComment{Comment}{/* }{ */}
\RestyleAlgo{ruled}

\newcommand{\ZZ}{\mathbb{Z}}
\newcommand{\PP}{\mathbb{P}}
\newcommand{\QQ}{\mathbb{Q}}

\newcommand{\smtx}[4]{\left(\begin{smallmatrix}#1&#2\\#3&#4\end{smallmatrix}\right)}

\renewcommand{\div}{\operatorname{div}}
\DeclareMathOperator{\Div}{Div}
\DeclareMathOperator{\length}{length}

\DeclareMathOperator{\PGL}{PGL}

\newcommand{\cF}{\mathcal{F}}
\newcommand{\cO}{\mathcal{O}}
\newcommand{\cR}{\mathcal{R}}
\newcommand{\eps}{\varepsilon}
\newcommand{\ab}{{\text{ab}}}

\newtheorem{theorem}{Theorem}[section]
\newtheorem{proposition}[theorem]{Proposition}
\newtheorem{lemma}[theorem]{Lemma}
\newtheorem{corollary}[theorem]{Corollary}

\theoremstyle{definition}
\newtheorem{definition}[theorem]{Definition}

\theoremstyle{remark}
\newtheorem{remark}[theorem]{Remark}

\numberwithin{equation}{section}

\renewcommand{\setminus}{\smallsetminus}